\documentclass{amsart}
\usepackage{enumerate}
\usepackage{amsthm}
\theoremstyle{plain}

\newtheorem*{theorem*}{Theorem}
\newtheorem{theorem}{Theorem}[section]
\newtheorem{lemma}[theorem]{Lemma}

\theoremstyle{definition}

\newtheorem{corollary}{Corollary}

\theoremstyle{remark}

\newtheorem*{remark*}{Remark}
\numberwithin{equation}{section}

\begin{document}


\title[(Semi-)Global Analytic Hypoellipticity]
{(Semi-)Global Analytic Hypoellipticity for a 
class of ``sums of squares" which fail to be locally analytic hypoelliptic}

\author{Gregorio Chinni}
\address{Fakult\"at f\"ur Mathematik, Oskar--Morgenstern--Platz 1, 1090 Vienna, Austria}
\email{gregorio.chinni@gmail.com}
\thanks{The author is supported by the Austrian Science Fund (FWF),
Lise-Meitner position, project no. M2324-N35.}


\date{\today}


\begin{abstract}
The global and semi-global analytic hypoellipticity
on the torus is proved for two classes of sums of squares operators, introduced in
\cite{ABM-TrC1-2016} and \cite{BM-TrC2-2016}, satisfying the H\"ormander condition and which fail 
to be neither locally nor microlocally analytic hypoelliptic.
\end{abstract}

\maketitle

\section{Introduction}
\noindent
Our aim, in this work, is to prove global and semi-global, i.e. local in some variables and global in others, 
analytic hypoellipticity on the torus for some models of sums of squares of vector fields with real valued   
and real analytic coefficients which satisfy H\"ormander condition, \cite{H67}.\\
In two recent papers, \cite{ABM-TrC1-2016} and \cite{BM-TrC2-2016}, Albano, Bove and Mughetti 
and  Bove and Mughetti produced and studied the first models of sums of squares operators not 
consistent with the (micro-)local Treves conjecture, \cite{Treves_Cj-2006}.
They showed that the sufficient part of the Treves' conjecture, for details on the subjet see 
\cite{Treves_Cj-2006}, does not hold neither locally nor microlocally. More precisely,
in \cite{ABM-TrC1-2016} the authors studied the model
\begin{equation}\label{Op-ABM}
P_{_{\!\!ABM}}(x,D)=
D_{1}^{2}\!+\!D_{2}^{2}\!+\!x_{1}^{2(r-1)}\!\left( D_{3}^{2}\!+\! D_{4}^{2}\right)+
x_{2}^{2(p-1)}D_{3}^{2}+x_{2}^{2(q-1)}D_{4}^{2}\, ,
\end{equation}
on $\Omega$, open neighborhood of the origin in $\mathbb{R}^{4}$, where $r$, $p$ and $q$ are 
positive integers such that $1 < r < p < q$. 
They showed that even if $P_{_{\!\!ABM}}$ has a single symplectic stratum, 
in meaning of the Poisson-Treves stratification, it is Gevrey hypoelliptic of order 
$s = r(q-1)[ q-1 +(r-1)(p-1)]^{-1}$ and not better.\\ 
In \cite{BM-TrC2-2016} the authors investigated the following operator
\begin{equation}\label{Op-BM}
P_{_{\!\!BM}}(x,D)\!=\!
D_{1}^{2}+ x_{1}^{2(r+\ell-1)}\left( D_{3}^{2}+D_{4}^{2}\right)
+x_{1}^{2\ell} \left( D_{2}^{2}+x_{2}^{2(p-1)}D_{3}^{2} + x_{2}^{2(q-1)}D_{4}^{2}\right)\!,
\end{equation}
on $\Omega$, open neighborhood of the origin in $\mathbb{R}^{4}$, with $1 < r < p < q$.
They proof that even if the codimention of the characteristic manifold of $P_{_{\!\!BM}}$
is 2 and the related stratification, in the sense of Treves, is made up by two symplectic strata the 
operator is not analytic hypoelliptic. It is Gevrey hypoelliptic of order 
$s = (\ell+r)(q-1)[(q-1)(\ell +1) +(r-1)(p-1)]^{-1}$ and not better.\\
Our purpose will be analyze the global and the semi-global analytic regularity on the four dimensional 
torus for two classes of operators which include as particular cases the global version of the operators
$P_{_{\!\!ABM}}$ and $P_{_{\!\!BM}}$.\\ 
Statement of the results.
%
\begin{theorem}\label{T1-G/L}
Let $P_{1}(x,D)= \sum_{j=1}^{6}X_{j}^{2}(x,D)$ be the operator given by
\begin{equation}\label{G/L-ABM}
D_{1}^{2}+D_{2}^{2}+ a^{2}(x_{1})\left( D_{3}^{2}+ D_{4}^{2}\right)
+b_{1}^{2}(x_{2})D_{3}^{2}+b_{2}^{2}(x_{2}) D_{4}^{2}
\end{equation}
on $ \mathbb{T}^{4}$ where $ a$, $b_{1}$ and $b_{2}$ are  real value real analytic functions not 
identically zero.
Then given any sub-interval 
$\mathcal{I} \subset \mathbb{T}^{2}_{x'}$, $x'=(x_{1},x_{2})$,
and given any
$u$ in $\mathcal{D}'(\mathcal{I} \times \mathbb{T}_{x''}^{2})$, $x''=(x_{3},x_{4})$, the condition 
$P_{1}u \in C^{\omega}(\mathcal{I} \times \mathbb{T}_{x''}^{2})$ implies
$u \in C^{\omega}(\mathcal{I} \times \mathbb{T}_{x''}^{2})$.
\end{theorem}
\vspace*{0.1em}
\begin{theorem}\label{T1.1-G/L}
Let $P_{1}(x,D)$ be as in (\ref{G/L-ABM}). Assume that $ a$, $b_{1}$ and $b_{2}$ are $0$ at zero
and the zero order at $x_{2}=0$ of $b_{2}$ is strictly greater than that of $b_{1}$.
Let $\mathcal{I}$ an open neighborhoods of $(x_{1},x_{2})=(0,0)$ and $\mathcal{U}$ a sub-interval
of $\mathbb{T}^{1}_{x_{4}}$.
Then if $P_{1}u=f$, with $f$ real analytic on 
$\mathcal{I}\times \mathbb{T}^{1}_{x_{3}}\times \mathcal{U}$ then $u$ is
analytic on $\mathcal{I}\times \mathbb{T}^{1}_{x_{3}}\times \mathcal{U}$.
\end{theorem}
\vspace*{0.3em}
A few remarks are in order.
\begin{itemize}
\item[(a)]If we take $ a(x_{1})= \!(\sin x_{1})^{r-1}$,
$b_{1}(x_{2})=\! (\sin x_{2})^{p-1}$ and $b_{2}(x_{2})=\!(\sin x_{2})^{q-1}$,
with $r$, $p$ and $q$ positive integers such that $1 < r < p < q$,
the operator $P_{1}(x,D)$ is the global version on the torus of the operator $P_{_{\!\!ABM}}$, 
(\ref{Op-ABM}), which is not local analytic hypoelliptic.
\item[(b)]We point out that if the zero order at $0$ of $b_{2}$ is equal than that of $b_{1}$ then the
operator $P_{1}(x,D)$ is microlocally anlytic hypoelliptic as showed in \cite{ABM-TrC1-2016}, hence
also global analytic hypoelliptic. Otherwise if the zero order at $0$ of $b_{2}$ is smaller than that of 
$b_{1}$ then the role of the directions $x_{3}$ and $x_{4}$ is exchanged,
i.e. the operator $P_{1}(x,D)$ is locally analytic hypoelliptic with respect to the variables $x_{1}$, 
$x_{2}$ and $x_{3}$ but globally analytic hypoelliptic with respect to the variable $x_{4}$.
\item[(c)]The operator $P_{1}$, (\ref{G/L-ABM}), belongs to the class studied by Cordaro and 
Himonas, \cite{Cordaro_Himonas-94}, therefore it is globally analytic hypoelliptic. 
\end{itemize}
\vspace*{0.2em}
\noindent
For completeness we recall the result proved in \cite{Cordaro_Himonas-94}.
\vspace*{0.2em}
\begin{theorem*}[\cite{Cordaro_Himonas-94}]
Let $P$ be a sum of squares operator, $P= \sum_{1}^{\nu}X_{j}$, on the torus
$\mathbb{T}^{N} = \mathbb{T}^{m}\times \mathbb{T}^{n}$
with variables, $(x',x'')$, $x' = (x_{1},\dots, x_{m})$, $x'' = (x_{m+1},\dots, x_{N})$ and
\begin{equation*}
X_{j} =\sum_{k=1}^{n}a_{jk}(x'')\frac{\partial}{\partial x_{m+k}} + \sum_{k=1}^{m}b_{jk}(x'')\frac{\partial}{\partial x_{k}}
\end{equation*}
are real vector fields with coefficients in $C^{\omega}\left(\mathbb{T}^{n}\right)$.
If the following two conditions hold:
\begin{itemize}
\item[{\bf (i)}] $X_{1}, \dots, X_{\nu}$ and their brackets of length at most $r$
span the tangent space at every point on $\mathbb{T}^{N}$, i.e. they satisfy the
H\"ormander condition,
\item[{\bf (ii)}] the vectors $\sum_{k=1}^{n}a_{jk}(x'')\partial_{x_{m+k}}$ span
$T_{x''}(\mathbb{T}^{n})$ for every $x'' \in \mathbb{T}^{n}$,
\end{itemize}
then the operator $P$ is globally analytic hypoelliptic on $\mathbb{T}^{N}$.
\end{theorem*}
\vspace*{0.3em}
\noindent
Next we look at the global and semi-global analytic regularity for operators which are a global version on 
the four dimensional torus of the operator studied in \cite{BM-TrC2-2016}.
\begin{theorem}\label{T2-Gl+G/L}
Let $P_{2}(x,D)= \sum_{j=1}^{6}X_{j}^{2}(x,D)$ be the operator given by
\begin{equation}\label{G-G/L-BM}
D_{1}^{2}+ a^{2}_{1}(x_{1})\left( D_{3}^{2}+ D_{4}^{2}\right)
+a^{2}_{2}(x_{1})\left(D_{2}^{2}+b_{1}^{2}(x_{2})D_{3}^{2}+b_{2}^{2}(x_{2}) D_{4}^{2}\right)
\end{equation}
on $\mathbb{T}^{4}$,
where $a_{j}$ and $b_{j}$, $j =1,2$, are real valued real analytic functions not identically zero. 
We have:
\begin{itemize}
\item[{\bf (i)}]Let $x_{1}^{0}$ be a common zero of $a_{1}$ and $a_{2}$ and assume that the
zero order at $x_{1}^{0}$ of $a_{2}$ is strictly greater than that of $a_{1}$. Let $\mathcal{I}_{1}$
an open neighborhood of $x_{1}^{0}$ and $\mathcal{I}_{2}$ a sub-interval of 
$\mathbb{T}^{1}_{x_{2}}$. 
The condition 
$P_{2}u \in C^{\omega}(\mathcal{I}_{1}\times \mathcal{I}_{2}\times\mathbb{T}^{2}_{x''})$,
$x''=(x_{3},x_{4})$, implies 
$u \in C^{\omega}(\mathcal{I}_{1}\times \mathcal{I}_{2}\times\mathbb{T}^{2}_{x''})$. 
\item[{\bf (ii)}]Let $(x_{1}^{0},x_{2}^{0})$ be a zero of $a_{i}$ and $b_{i}$, $i=1,2$, and assume 
that the zero order at $x_{1}^{0}$ of $a_{1}$ is strictly greater than that of $a_{2}$ and that the 
zero order at $x_{2}^{0}$ of $b_{2}$ is strictly greater than that of $b_{1}$.
Let $\mathcal{I}$ an open neighborhood of $(x_{1}^{0},x_{2}^{0})$ and $\mathcal{U}$ a sub-interval 
of $\mathbb{T}^{1}_{x_{4}}$. 
The condition 
$P_{2}u \in C^{\omega}(\mathcal{I}\times \mathbb{T}^{1}_{x_{3}}\times\mathcal{U})$,
implies $u \in C^{\omega}(\mathcal{I}\times \mathbb{T}^{1}_{x_{3}}\times\mathcal{U})$. 
\end{itemize}
\end{theorem}
\noindent
Moreover, with the aid of the partition of unity we have:
%
\begin{corollary}
\textit{Let $P_{2}(x,D)$ be as in (\ref{G-G/L-BM}). Then the operator $P_{2}(x,D)$ is globally analytic 
hypoelliptic on $\mathbb{T}^{4}$.}
\end{corollary}
\vspace*{0.3em}
Some remarks are in order.
\begin{itemize}
\item[(a)]If we take $a_{1}(x_{1})= \!(\sin x_{1})^{r+\ell-1}$, $ a_{2}(x_{1})= \!(\sin x_{1})^{\ell}$
$b_{1}(x_{2})=\! (\sin x_{2})^{p-1}$ and $b_{2}(x_{2})=\!(\sin x_{2})^{q-1}$,
with $r$, $p$, $q$ and $\ell$ positive integers such that $1 < r < p < q$,
the operator $P_{2}(x,D)$ is the global version on the torus of the operator $P_{_{\!\!BM}}$, 
(\ref{Op-BM}), which is not local analytic hypoelliptic.
\item[(b)]The operator $P_{2}$, (\ref{G-G/L-BM}), does not belong to the class studied by Cordaro 
and Himonas, \cite{Cordaro_Himonas-94}. 
\item[(c)]Theorem \ref{T2-Gl+G/L}-\textbf{(ii)}: if the zero order at $x_{2}^{0}$ of $b_{2}$ is equal than that of $b_{1}$ then the
operator $P_{2}(x,D)$ is microlocally anlytic hypoelliptic as showed in \cite{BM-TrC2-2016}, hence
also global analytic hypoelliptic. Otherwise if the zero order at $x_{2}^{0}$ of $b_{2}$ is smaller than 
that of $b_{1}$ then the role of the directions $x_{3}$ and $x_{4}$ is exchanged,
i.e. the operator $P_{2}(x,D)$ is locally analytic with respect to the variables $x_{1}$, 
$x_{2}$ and $x_{3}$ and globally analytic with respect to the variable $x_{4}$.
\end{itemize}
\vspace*{1em}
\begin{remark*}
The results obtained are ``consistent" with the global version of the Treves conjecture,
\cite{Treves_Cj-2006}. In both case the (semi-)global analytic hypoellipticity  is due to the fact that
the bicharacteristic leaf of the missing stratum, see Remark 2.1\cite{BM-TrC2-2016}, 
$\tilde{\Sigma}= \left\lbrace (0,0,x_{3},x_{4}; 0,0,0,\xi_{4})| \xi_{4}\neq 0\right\rbrace$
is compact. 
 
\end{remark*}
%
\vspace*{1em}
\noindent
The interest in this work was inspired by the seminal works of Cordaro and Himonas, 
\cite{Cordaro_Himonas-94} and \cite{Cordaro_Himonas-98}, and Tartakoff, \cite{Tartakoff-96}.
To obtain the results we will follows the ideas in  \cite{Cordaro_Himonas-94}, proof of the Theorem 
\ref{T1-G/L}, and the ideas in \cite{Tartakoff-96}, proof of the Theorems \ref{T1.1-G/L} and
\ref{T2-Gl+G/L}.
\section{Proof of the theorem \ref{T1-G/L}}
Without loss of generality we assume that $x'=(0,0)$ is a zero for the functions $a$, $b_{1}$ and 
$b_{2}$,
 $\mathcal{I} \doteq \mathcal{I}_{1}\times\mathcal{I}_{2}=\, ]-\delta_{1},\delta_{1}[\times ]-
 \delta_{2},\delta_{2}[$, $\delta_{i}>0$, $a(x_{1})\neq 0 $ for 
 $x_{1}\in \mathcal{I}_{1}\setminus\lbrace 0 \rbrace$ and $b_{j}(x_{2})\neq 0 $ for 
 $x_{2}\in \mathcal{I}_{2}\setminus\lbrace 0\rbrace$, $j=1,2$. 
By H\"ormander theorem, \cite{H67},  $P_{1}$ is hypoelliptic, therefore we can assume 
$u \in C^{\infty}(\mathcal{I}\times \mathbb{T}^{2}_{x''})$. 
Taking the Fourier transform with respect to $x''$ we have
\begin{equation*}
\widehat{P_{2}u}(x',\xi'')\!=\!
\widehat{D_{1}^{2}u}(x',\xi'')+ \widehat{D_{2}^{2}u}(x',\xi'')+
\left[a^{2}(x_{1})|\xi''|^{2}\!\!+b_{1}^{2}(x_{2})\xi_{3}^{2}+b_{2}^{2}(x_{2})\xi^{2}_{4}\right]\!
\widehat{u}(x',\xi'').
\end{equation*} 
We multiply by $\bar{\widehat{u}}$ and integrate in $\mathcal{I}$: 
\begin{align*}
\int_{\mathcal{I}}\widehat{P_{1}u}(x',\xi'')\bar{\widehat{u}}(x',\xi'')&dx'=
\int_{\mathcal{I}}\left((\widehat{\partial_{1}^{2}u})(x',\xi'')
+(\widehat{\partial_{2}^{2}u})(x',\xi'')\right)\bar{\widehat{u}}(x',\xi'')dx'\\
&+\int_{\mathcal{I}}\left[a^{2}(x_{1})|\xi''|^{2}
+b_{1}^{2}(x_{2})\xi_{3}^{2}+b_{2}^{2}(x_{2})\xi^{2}_{4}\right]|\widehat{u}(x',\xi'')|^{2}dx'.
\end{align*}
We have
\begin{align}\label{eq-P1}
&\int_{\mathcal{I}}\!\left[a^{2}(x_{1})|\xi''|^{2}+b_{1}^{2}(x_{2})\xi_{3}^{2}
+b_{2}^{2}(x_{2})\xi^{2}_{4}\right]|\widehat{u}(x',\xi'')|^{2}dx'
+\int_{\mathcal{I}}\!|\widehat{u}_{x_{1}}(x',\xi'')|^{2}dx'\\
&
\nonumber
+\int_{\mathcal{I}}|\widehat{u}_{x_{2}}(x',\xi'')|^{2}dx'
=\int_{\mathcal{I}_{1}}\!\!\widehat{u}_{x_{2}}(x',\xi'')\bar{\widehat{u}}(x',\xi'')
\Big|_{x_{2}=-\delta_{2}}^{x_{2}=\delta_{2}}dx_{2}\\
&
\nonumber
+\int_{\mathcal{I}_{2}}\!\!\widehat{u}_{x_{1}}(x',\xi'')\bar{\widehat{u}}(x',\xi'')
\Big|_{x_{1}=-\delta_{1}}^{x_{1}=\delta_{1}}dx_{2}
\int_{\mathcal{I}}\widehat{P_{2}u}(x',\xi'')\bar{\widehat{u}}(x',\xi'')dx',
\end{align}
where $\displaystyle\widehat{u}_{x_{i}} \doteq \displaystyle\widehat{\partial_{i}u}$, $i=1,2$.
Since $Pu \in C^{\omega}(\mathcal{I}\times\mathbb{T}^{2}_{x''})$ and $P$ is elliptic away
from $(0,0)$ we can estimate the left hand side of the above equality by
$C\displaystyle e^{-\varepsilon|\xi''|}$, with $C$ and $\varepsilon$ suitable positive constants.\\
In order to complete the proof we need of an analogous, in two variable, of the Lemma 4.1 in 
\cite{Cordaro_Himonas-94}. 
\begin{lemma}
For $f\in C^{\infty}(\overline{\mathcal{I}})$ let 
\begin{equation*}
\|f\|_{g}^{2}=\int_{\mathcal{I}} g^{2}(x')|f(x')|^{2}dx' +\int_{\mathcal{I}}|(\partial_{1}f)(x')|^{2}
+|(\partial_{2}f)(x')|^{2}\,dx',
\end{equation*}
where $g$ is a real analytic function on $\mathcal{I}$ not identically zero such that $g(0)=0$
and $g(x')\neq 0 $ for every $x'\in \overline{\mathcal{I}}\setminus \lbrace 0\rbrace$.  
Then there is a positive constant depending on $g$ such that
\begin{align}\label{Es-L1}
\|f\|^{2}_{0}\leq C\|f\|^{2}_{g}.
\end{align}
\end{lemma}
\begin{proof}
We have
\begin{align*}
f(x_{1},x_{2})= f(y_{1},y_{2}) + \int_{y_{1}}^{x_{2}}(\partial_{2}f)(y_{1},t_{2})\,dt_{2}
+\int_{y_{1}}^{x_{1}}(\partial_{1}f)(t_{1},x_{2})\, dt_{1}.
\end{align*}
Since $g(y') \neq 0$ for every $y'\in \overline{\mathcal{I}}\setminus \lbrace 0\rbrace$, there exists
$\alpha > 0$ on $ ]\frac{\delta_{1}}{2},\delta_{1}\times ]\frac{\delta_{1}}{2},\delta_{2}[$ such
that $g^{2}(y')> \alpha^{2}$, we have
\begin{align*}
|f(x_{1},x_{2})|^{2}
\!\leq \! C\! \left(
\int_{\mathcal{I}}\!g^{2}(y')|f(y')|^{2}\,dy' 
\!+\! \int_{\mathcal{I}}\!\!(\partial_{2}f)(y_{1},t_{2})\,dt_{2}\,dy_{1}
\!+\!\int_{-\delta_{1}}^{\delta_{1}}\!\!\!\!(\partial_{1}f)(t_{1},x_{2})\, dt_{1}
\right)\!,
\end{align*}
where $C$ depends on $\alpha$, $\delta_{1}$ and $\delta_{2}$.
By integrating the above inequality on $\mathcal{I}$ with respect to $x'$ we obtain (\ref{Es-L1}).
\end{proof}
\noindent
Applying the above Lemma with $f(x')=\widehat{u}(x',\xi'')$ and
$g^{2}(x')= a^{2}(x_{1})|\xi''|^{2}+b_{1}^{2}(x_{2})\xi_{3}^{2} +b_{2}^{2}(x_{2})\xi^{2}_{4}$,
$\xi''\neq 0$, we can estimate from below the right hand side of (\ref{eq-P1}), equal to 
$\|\widehat{u}(\cdot,\xi'')\|_{g}^{2}$, with $\| \widehat{u} (\cdot, \xi'')\|^{2}$ . We have
\begin{equation}
\| \widehat{u}(\cdot,\xi'')\|_{0} \leq C e^{-\varepsilon|\xi''|}, \quad \xi''\in \mathbb{Z}^{2},
\end{equation}
where $C$ and $\varepsilon$ are suitable positive constants.\\
Let $\phi \in C^{\infty}_{0}(\mathcal{I})$ with $\phi \equiv 1$ in  $\mathcal{I}_{1}$,
$\mathcal{I}_{1}$ neighborhood of the origin compactly contained in $\mathcal{I}$.
Let $u_{1}(x',x'')=\phi(x')u(x',x'')$, we have
\begin{align*}
|\widehat{u}_{1}(\xi',\xi'')| =\Big |\int_{\mathbb{T}^{2}_{x'}}
e^{-i\langle x',\xi'\rangle} \widehat{u}_{1}(x',\xi'') \,\,dt\,\Big |
C_{1}\leq \| \widehat{u}(\cdot,\xi'')\|_{0} \leq C_{2} e^{-\varepsilon_{1}|\xi|},
\end{align*}
for every $(\xi',\xi'')\in \mathbb{Z}^{4}$ with $|\xi''|\neq 0$ and $|\xi'|< c |\xi''|$, $c >0$.
This shows that the points of the form 
$(x',x'', \xi',\xi'' ) \in T^{*}(\mathcal{I}\times \mathbb{T}^{2}_{x''})\setminus \lbrace 0\rbrace$
with $\xi''\neq 0$ and $|\xi'|< c |\xi''|$ do not belong to $WF_{a}(u)$, the analytic wave front set of 
$u$. Therefore there is no points in $\textit{Char}(P_{1})$, the characteristic variety of $P_{1}$,
which belong to $WF_{a}(u)$.
By the Theorem 8.6.1 in \cite{H80-book1} we conclude that the analytic wave front set of $u$ is 
empty.
\section{Proof of the theorem \ref{T1.1-G/L}}
\noindent
Since the vector fields $X_{1}, \,\dots, \, X_{6}$ satisfy the H\"ormander condition, \cite{H67},
$P_{1}$ is hypoelliptic. Furthermore with the aid of the partition of unity the operator $P_{1}$
satisfies the following subelliptic a priori estimate:
\begin{equation}\label{SubP1}
\| u\|_{\frac{1}{r}}^{2} + \sum_{j=1}^{6}\|X_{j}u\|^{2}
\leq C|\langle P_{1}u, u\rangle| +C^{N+1}\|u\|^{2}_{_{-N}},
\end{equation}
for every $N \in \mathbb{Z}_{+}$. Here $u$ is a smooth function on 
$\mathcal{I}\times \mathbb{T}_{x_{3}}\times \mathcal{U}$ with compact support with respect to 
$x_{1}$, $x_{2}$ and $x_{4}$. $\|\cdot\|_{s}$ denotes the Sobolev norm of order $s$ and $r$ the 
length of the iterated commutator such that the vector fields, their commutators, their triple 
commutators etcetera up to the commutators of length $r$ generate a Lie algebra of dimension
equal to that of the ambient space.  More precisely $r-1$ is the minimum between the zero order
at $0$ of $a$ and that at $0$ of $b_{1}$.
The above estimate was proved first by H\"ormander in \cite{H67} 
for a Sobolev norm of order $r^{-1}+\varepsilon$ and up to order $r^{-1}$ subsequently by
Rothschild and Stein \cite{RS}.\\
To achieve the result, we want show the analytic growth of high order derivatives of the solutions
in $L^{2}$-norm. As a matter of fact we estimate a suitable localization of a high derivative of the 
solutions using (\ref{SubP1}).\\
Let $\phi_{N}(x_{1},x_{2},x_{4})$ be a cutoff function of Ehrenpreis-H\"ormander type:
$\phi_{N}$ in $C^{\infty}_{0}\left(\mathcal{I}\times \mathcal{U}\right)$ non negative
such that $ \phi_{N} \equiv 1 $ on $\mathcal{U}_{0} $, $ \mathcal{U}_{0} $ neighborhood of the 
origin compactly contained in $\mathcal{I}\times \mathcal{U} $, and exist a constant $ C $ such that 
for every $  |\alpha|  \leq 2 N $ we have $ | D^{\alpha} \phi_{N}(x) | \leq C^{ \alpha +1} N^{\alpha  }$,
$\alpha \in \mathbb{Z}^{3}$.\\
We may assume that $ \phi_{N} $ is independent of  the $ x_{1} $ and $x_{2}$-variable:
every $ x_{1} $, $x_{2}$-derivative landing on $ \phi_{N} $ would leave a cut off function 
supported where $ x_{1} $ or  $x_{2}$ is bounded away from zero, where the operator is elliptic.
As in \cite{Tartakoff-96}, to gain the result we have to show the analytic growth of
 $\phi_{N} D_{j}^{N}u$, $j=1,\, 2,\, 3,\, 4$, via (\ref{SubP1}).
It will be sufficient analyze the direction $D_{4}$.
As matter of fact $D_{3}$ commutes with $P_{1}$ and, moreover, following the same strategy 
employed to analyze the case $D_{4}$, we can transform powers of $D_{1}$ and $D_{2}$ in powers 
of $D_{3}$ and $D_{4}$.\\
We replace $u$ in (\ref{SubP1}) by $\phi_{N} D_{4}^{N}u$. We have
\begin{align}\label{SubP1c}
\| \phi_{N} D_{4}^{N}u\|_{1/r}^{2} + \sum_{j=1}^{6}\|X_{j}\phi_{N} D_{4}^{N}u\|^{2}
\leq C \, |\langle   & P_{1}\phi_{N} D_{4}^{N}u, \phi_{N} D_{4}^{N}u\rangle|
\\
&\nonumber\qquad 
+C^{N+1}\|\phi_{N} D_{4}^{N}u\|_{_{-N}}.
\end{align} 
The last term on the right hand side gives analytic growth. The scalar product:
\begin{align*}
& \langle  \phi_{N}  D_{4}^{N} P_{1}  u, \phi_{N}  D_{4}^{N} u \rangle
+ 
\sum_{j=1}^{6} 
\langle \lbrack X_{j}^{2}, \phi_{N}  D_{4}^{N} \rbrack u, \phi_{N}  D_{4}^{N}u \rangle
\\
&\quad
= 2 \sum_{j=1}^{6}
 \langle  \lbrack X_{j}, \phi_{N}  D_{4}^{N} \rbrack u, X_{j}\phi_{N}  D_{4}^{N}u \rangle
+\sum_{j=1}^{6} 
 \langle \lbrack X_{j}, \lbrack X_{j}, \phi_{N}  D_{4}^{N} \rbrack \rbrack u, \phi_{N}  D_{4}^{N} u 
 \rangle
 \nonumber
\\
&\hspace*{23em}
+  \langle  \phi_{N} D_{4}^{N} P_{2}  u, \phi_{N} D_{4}^{N} u \rangle.
\nonumber
\end{align*}
With regard to the last scalar product on the right hand side we have
\begin{align*}
&|\langle  \phi_{N} D_{4}^{N} P_{1}  u, \phi_{N} D_{4}^{N} u \rangle|
\leq 
\left(\frac{1}{2C}\right)
 \|\phi_{N} D_{4}^{N} u\|^{2}_{\frac{1}{r}}
+\left(2C\right)^{rN}\|\varphi_{i}\phi_{N}D_{4}^{N}u\|^{2}_{-N} 
\\
&\hspace*{27em}+ \|\phi_{N} D_{4}^{N} P_{2}  u\|^{2}.
\end{align*}
The last two terms give analytic growth, $P_{1}u\in C^{\omega}$; the first one can be absorbed on 
the left hand side of (\ref{SubP1c}).\\
Since $\phi_{N}$ depend only by $x_{4}$ we have to analyze the commutators with,  $X_{3}$, and 
$X_{6}$. Since the same strategy can be used to handle the case involving $X_{3}$ and $X_{6}$, we 
will give the details only of the case $X_{3}$. We have
\begin{align}\label{es-X_3-P1}
&
2| \langle  \lbrack X_{3}, \phi_{N} D_{4}^{N} \rbrack u, X_{3}\phi_{N} D_{4}^{N}u \rangle|
+ 
| \langle \lbrack X_{3},\lbrack X_{3}, \phi_{N} D_{4}^{N} \rbrack \rbrack u, \phi_{N} D_{4}^{N} u 
\rangle| 
\\
&\nonumber\quad
= 2 | \langle a_{1} \phi^{(1)}_{N} D_{4}^{N} u, X_{3}\phi_{N} D_{4}^{N}u \rangle |
+ | \langle  a_{1}^{2}  \phi^{(2)}_{N} D_{4}^{N} u, \phi_{N} D_{4}^{N} u \rangle |.
\end{align}
The first term can be estimate by
\begin{align*}
| \langle a_{1} \phi^{(1)}_{N} D_{4}^{N} u, X_{3}\varphi_{N} D_{4}^{N}u \rangle |
\leq  \sum_{j=1}^{ N}  C_{j} \| X_{3} \phi^{(j)}_{N} D_{4}^{N - j} u \|^{2}
&+  \sum_{j=1}^{ N+1}\!\frac{1}{C_{j}} \| X_{3}\phi_{N} D_{4}^{N}u \|^{2}
\\
&\nonumber\qquad
+  C_{N+1}\| \phi^{(N+1)}_{N}  u \|^{2} ,
\end{align*}
The constants $ C_{j}$ are arbitrary, we make the choice $ C_{j} = \varepsilon^{-1} 2^{j} $,
$ \varepsilon $ suitable small positive constant. The terms of the form
$ C_{j}^{-1} \| X_{3} \phi_{N} D_{4}^{N} u\|^{2} $ can be absorbed on the right hand side of 
(\ref{SubP1c}). The last term gives analytic growth.
Finally we  observe that the terms in the first sum have the same form as 
$ \| X_{3} \phi_{N} D_{4}^{N } u \|^{2} $ where one or more $ x_{4} $-derivatives have been shifted 
from $ u $ to $ \phi_{N} $; on these terms we can take maximal advantage from the 
sub-elliptic estimate restarting the process. \\
With regard to the second term on the right hand side of  (\ref{es-X_3-P1}) we have
\begin{align*}
| \langle  a_{1}^{2}\phi^{(2)}_{N} D_{4}^{N} u, \phi_{N} D_{4}^{N} u \rangle |
&
\leq 
\frac{1}{2N^{2}} \| X_{3} \phi^{(2)}_{N} D_{4}^{N -1} u \|^{2} 
+ \frac{N^{2}}{2} \| X_{3} \phi_{N} D_{4}^{N-1} u \|^{2}   
\\
&\phantom{=}
+| \langle  a_{1}\phi^{(2)}_{N} D_{3}^{N-1} u, X_{3} \phi^{(1)}_{N} D_{4}^{N-1} u \rangle |
\\
&\phantom{=}
+ | \langle N^{-1} a_{1}  \phi^{(3)}_{N} D_{4}^{N-1} u, N X_{3} \phi_{N} D_{4}^{N-1} u \rangle |
\\
&\phantom{=}
+| \langle a_{1}^{2}\phi^{(3)}_{N} D_{4}^{N-1} u, \phi^{(1)}_{N} D_{4}^{N-1} u \rangle |.
\end{align*}
On the first two  terms we can take maximal advantage from the sub-elliptic estimate restarting the 
process. The `` weight'' $ N $ introduced above helps to balance the number of $ x_{4}$-derivatives 
on $ u $ with the number of derivatives on $ \phi_{N} $, we take the factor $ N $ as a derivative on
$ \phi_{N} $ and $ N^{-1} \phi^{(2)}_{N} $ as $ \phi^{(1)}_{N}$. 
The second and the third terms have the same form of the first term on the right hand side of
(\ref{es-X_3-P1}), the third one with the help of the weight $ N $; we can handled 
both in the same way. The last term is the same of the left hand side in which one $ x_{4}$-derivative
has been shifted from  $ u $ to $ \phi_{N} $ on both side. 
Restarting the process we can estimate the left hand side of the above inequality by
\begin{align*}
&\frac{1}{2N^{2}}\sum_{j=1}^{N}\|X_{3}\phi_{N}^{(j+1)}D^{N-j}_{4}u\|^{2}
+
\frac{N^{2}}{2}\sum_{j=1}^{N}\|X_{3}\phi_{N}^{(j-1)}D^{N-j}_{4}u\|^{2}
\\
&
+\sum_{j=1}^{N}\sum_{\ell=j}^{N} \|X_{3}\phi_{N}^{(N-\ell+j+1)}D^{\ell-j}_{4}u\|^{2}
+\sum_{j=1}^{N}2^{j+1}\|X_{3}\phi_{N}^{(N-j+1)}D^{j-1}_{4}u\|^{2}
\\
&
+\frac{1}{N^{2}}
\sum_{j=1}^{N}\sum_{\ell=j}^{N} \|X_{3}\phi_{N}^{(N-\ell+j+2)}D^{\ell-j}_{4}u\|^{2}
+N^{2}\sum_{j=1}^{N}C^{j+1}\|X_{3}\phi_{N}^{(N-j)}D^{j-1}_{4}u\|^{2}\\
&
+2^{N+1}\left( \|a_{1}\phi_{N}^{(N+1)}u\|^{2} +\|a_{1}\phi_{N}^{(N+2)}u\|^{2}\right)
+\|a_{1}\phi_{N}^{(N)}u\|^{2}. 
\end{align*}
The last terms give analytic growth, the others, in the sums, have the same form as 
$\|X_{3}\phi_{N}D^{N}_{4}u\|^{2}$, we can restart  the process without the help of the sub-ellipticity.  
\noindent
Therefore at any step of the process we obtain or terms which give analytic growth or terms from 
which we can take maximum advantage from the sub-elliptic estimate.
We can conclude 
\begin{align*}
\| \phi_{N} D_{4}^{N}u\|_{1/r}^{2} + \sum_{j=1}^{6}\|X_{j}\phi_{N} D_{4}^{N}u\|^{2}
\leq 
C^{N+1} (N)^{2N},
\end{align*} 
where $C$ is independent by $N$ but depends on $u$ and $a_{1}$. This conclude the proof.
%
\section{Proof of the theorem \ref{T2-Gl+G/L}}
\noindent
\paragraph{\textbf{Part (i), Theorem \ref{T2-Gl+G/L}}}
Without loss of generality we assume that $x_{1}^{0}=0$ and 
$\mathcal{I}_{1}\times\mathcal{I}_{2}$ is a neighborhood of the point $x'=(0,0)$.
Since the vector fields $X_{1}, \dots\, , X_{6}$ satisfy the H\"ormander condition
$P_{2}$ is hypoelliptic, it has the following sub-elliptic estimate:
\begin{align}\label{Sub-P2-1}
\| u\|_{1/r}^{2} + \sum_{j=1}^{6}\|X_{j}u\|^{2}
\leq C\left( |\langle P_{2}u, u\rangle| +\|u\|^{2}_{0}\right),
\end{align} 
where $u$ is a smooth function on 
$\mathcal{I}_{1}\times\mathcal{I}_{2} \times \mathbb{T}^{2}_{x''}$ with compact support with
respect to $x'$. Here $r-1$ is the zero order at $0$ of $a_{2}$.\\
\noindent
As in the proof of the Theorem \ref{T1.1-G/L} the result will be achieved via the $L^{2}$ estimate of 
suitable localizations of high derivatives of the solutions. 
Even if not strictly necessary in this situation we will follow a little bit different strategy which will 
involve the partition of unity, as done in \cite{Tartakoff-96}. This more general approach would allow 
us, without particular technical efforts, to extend the results to a more general setting in which the 
two-dimensional torus is replaced by a compact real analytic manifold, $M$, without boundary and the 
vector fields $D_{3}$ and $D_{4}$ are replaced by a couple of real analytic vector fields $X_{3} $ and 
$X_{4}$ on $M$ such that they span $TM$ at each point.\\
\noindent
Let $\phi_{N}(x_{2})$ be a cutoff function of Ehrenpreis-H\"ormander type. $\phi_{N}$ is taken
independent of the $x_{1}$-variable since every $ x_{1}$-derivative landing on $ \phi_{N} $ would 
leave a cut off function supported where $ x_{1} $ is bounded away from zero, where the operator
is elliptic.\\
\noindent
Let $\lbrace\mathcal{V}_{j}\rbrace$ be a finite covering of $\mathbb{T}^{2}_{x''}$, $j=1,\dots,k$, 
and $\lbrace \varphi_{j}\rbrace$ a partition of unity subordinate to to such a cover, 
$\varphi_{j}\in C^{\infty}_{0}(\mathcal{V}_{j})$, $\varphi_{j}\geq 0$ and 
$\sum \varphi_{j} =1$.\\
\noindent
We replace $ u $ in $(\ref{Sub-P2-1})$ by $\varphi_{j}(x_{3},x_{4})\phi_{N}(x_{2}) D_{2}^{N}u$.
We have
\begin{align}\label{Sb_P2-1.1}
&\|\varphi_{j}\phi_{N} D_{2}^{N} u\|^{2}_{\frac{1}{r}}
+ \sum_{i=1}^{6}\|X_{i}\varphi_{j}\phi_{N} D_{2}^{N}u\|^{2}_{0}
\leq C  |\langle P_{2}\varphi_{j}\phi_{N} D_{2}^{N} u, \varphi_{j}\phi_{N} D_{2}^{N}u\rangle |
\\
&
\hspace*{21em} + C^{N+1}\|\varphi_{j}\phi_{N}D_{2}^{N}u\|^{2}_{-N}.
\nonumber
\end{align}
The last term on the right hand side gives analytic growth. 
As done in the proof of the Theorem \ref{T1.1-G/L} we have to handle the scalar product on the right
hand side, more precisely we have to study terms of type
\begin{align*}
\langle [X_{i}, \varphi_{j}\phi_{N} D_{2}^{N}] u, X_{i}\varphi_{j}\phi_{N} D_{2}^{N} u\rangle,
\hspace*{0.5em}
\langle[X_{i}, [X_{i}, \varphi_{j}\phi_{N} D_{2}^{N}] ]u, \varphi_{j}\phi_{N} D_{2}^{N} u\rangle,
\end{align*}
$i=2,\dots, 6$. The case $X_{4}= a_{2}(x_{1})D_{2}$ can handled following the same strategy used 
in the proof of Theorem \ref{T1.1-G/L}, see (\ref{es-X_3-P1}), in this case we can take maximal 
advantage from the sub-elliptic estimate, therefore it gives analytic growth.  
Concerning the other cases it is sufficient study the case $X_{2} = a_{1}(x_{1})D_{3}$, 
the remaining cases can be handled following the same strategy
\footnote{We remark that the terms involving the fields $X_{5}$ and $X_{6}$ could be handled taking 
maximum advantage from the sub-elliptic estimate, this could be done choosing a partition of unity 
subordinate to the cover, whose elements are cutoff functions of Ehrenpreis-H\"ormander type.}.
 We have to estimate
\begin{align}\label{Est-X_2}
2|\langle a_{1} \varphi_{j}^{(1)}\phi_{N} D_{2}^{N}u, X_{2}\varphi_{j}\phi_{N} D_{2}^{N} u
\rangle|
+
|\langle a_{1}^{2} \varphi_{j}^{(2)} \phi_{N} D_{2}^{N}u, \varphi_{j}\phi_{N} D_{2}^{N} u\rangle|
\doteq I_{1}+I_{2},
\end{align}
where $\varphi_{j}^{(\ell)}=\partial_{3}^{^{\ell}}\varphi_{j}$. Here we can not take maximum 
advantage from the sub-elliptic estimate. In the local case would be this term which would give Gevrey 
growth. The argument that we will use to handle these two terms is the reason because the results is 
global and not local with respect to the $x_{3}$-variable.  We have
\begin{align}\label{Est-I_1}
I_{1}&\leq 4C \|a_{1}\varphi_{j}^{(1)} \phi_{N}D_{2}^{N}u\|^{2}
+ \frac{1}{4C}\|X_{2}\varphi_{j} \phi_{N}D_{2}^{N}u\|^{2} \\
&\nonumber
\leq 4C \|a_{1}\|_{\infty}^{2}\sup_{j} \|\varphi_{j}^{(1)}\|_{\infty}^{2}\|\phi_{N}D_{2}^{N}u\|^{2}
+\frac{1}{4C}\|X_{2}\varphi_{j} \phi_{N}D_{2}^{N}u\|^{2}\\
&\nonumber
\leq 4C  \|a_{1}\|_{\infty}^{2}\sup_{j} \|\varphi_{j}^{(1)}\|_{\infty}^{2} C_{1} 
\sum_{j=1}^{k} \|\varphi_{j} \phi_{N}D_{2}^{N}u\|^{2}_{1/r}
+ \frac{1}{4C}\|X_{2}\varphi_{j} \phi_{N}D_{2}^{N}u\|^{2} \\
&\nonumber
\hspace*{9em}
+ 4C   \|a_{1}\|_{\infty}^{2}\sup_{j} \|\varphi_{j}^{(1)}\|_{\infty}^{2} C_{1}^{-rN}
\sum_{j=1}^{k} \|\varphi_{j} \phi_{N}D_{2}^{N}u\|^{2}_{-N},
\end{align}
where the constant $C_{1}$ is arbitrary. The second term on the right hand side can be absorbed on 
the left hand side of (\ref{Sb_P2-1.1}), the last one gives analytic growth. The term $I_{2}$ in 
(\ref{Est-X_2}) can be estimate by
\begin{align}\label{Est-I_2}
\phantom{I_{2}}&
 \|a_{1}^{2}\|_{\infty}^{2}\sup_{i} \|\varphi_{i}^{(2)}\|_{\infty}^{2}\left( C_{2} 
\sum_{j=1}^{k} \|\varphi_{j} \phi_{N}D_{2}^{N}u\|^{2}_{1/r}
+ C_{2}^{-rN}
\sum_{i=1}^{k} \|\varphi_{j} \phi_{N}D_{2}^{N}u\|^{2}_{-N}\right)\\
&\nonumber
\hspace*{11em}
+\frac{1}{2C}  \|\varphi_{j} \phi_{N}D_{2}^{N}u\|^{2}_{1/r}
+  (2C)^{rN} \|\varphi_{j} \phi_{N}D_{2}^{N}u\|^{2}_{-N},
\end{align}
where the constant $C_{2}$ is arbitrary. The last term gives analytic growth and the second to last  
can be absorbed on the left hand side of (\ref{Sb_P2-1.1}).
Summing (\ref{Sb_P2-1.1}) over $j$ and choosing $C_{1}$ and $C_{2}$ small enough so that the first 
term in (\ref{Est-I_1}) and the first one in (\ref{Est-I_2}) can be absorbed on the left,
we can conclude 
\begin{align*}
\| \phi_{N} D_{2}^{N}u\|_{1/r}^{2} + \sum_{j=1}^{6}\|X_{j}\phi_{N} D_{2}^{N}u\|^{2}
\leq 
C^{N+1} (N)^{2N},
\end{align*} 
where $C$ is independent by $N$ but depends on $u$. This conclude the proof.
%
\noindent
\paragraph{\textbf{Part (ii), Theorem \ref{T2-Gl+G/L}}}
We can  assume that $(x_{1}^{0},x_{2}^{0})=(0,0)$ and $\mathcal{U}$ is
an open neighborhood of the zero. 
Since the vector fields satisfy the H\"ormander condition at the step $r$, for some 
$r \in \mathbb{Z}_{+}$,  the following a priori estimate holds:
\begin{equation}\label{Sub_P-2}
\|u\|^{2}_{\frac{1}{r}}+ \sum_{j=1}^{6}\|X_{j}u\|^{2}_{0}
\leq C \left( |\langle P_{2}u, u\rangle| + C^{N}\|u\|^{2}_{-N}\right),
\quad \forall\, N\in \mathbb{Z}_{+}.
\end{equation}
Here $u$ is a smooth function on 
$\mathcal{I}\times \mathbb{T}_{x_{3}}\times \mathcal{U}$ with compact support with respect to 
$x_{1}$, $x_{2}$ and $x_{4}$. 
The result is obtained via estimate of suitable localization of high derivatives, that is estimating
$\phi_{N}(x_{4})D_{4}^{N}u$ through (\ref{Sub_P-2}). We will not give the details since the proof can be 
easily archived following the same strategies used in the proofs of the Theorem \ref{T1.1-G/L} and
Theorem \ref{T2-Gl+G/L}-(i). We only remark that the cutoff function of Ehrenpreis-H\"ormander type,
$\phi_{N}$,  can be assumed independent of  the $ x_{1} $ and $x_{2}$-variable:
every $ x_{1} $-derivative landing on $ \phi_{N} $ would leave a cut off function supported where
$ x_{1} $ is bounded away from zero, where the operator is elliptic; every $ x_{2}$-derivative
landing on $ \phi_{N} $ would leave a cut off function supported where
$ x_{2} $ is bounded away from zero, in this region the operator $P_{2}$ behaves like
the operator $D_{1}^{2} +a^{2}_{2}(x_{1})\left(D_{2}^{2}+D_{3}^{2}+D_{4}^{2}\right)
+ a^{2}_{1}(x_{1})\left( D_{3}^{2}+ D_{4}^{2}\right)$, which is (micro-)locally analytic hypoelliptic,
therefore (semi-)globally analytic hypoelliptic.
%
\bibliographystyle{amsplain}

\begin{thebibliography}{10}

\bibitem{ABM-TrC1-2016}
{\sc P. Albano and A. Bove and M. Mughetti},
{\it Analytic Hypoellipticity for Sums of Squares and the Treves Conjecture},
Preprint, http://arxiv.org/abs/ 1605.03801, 2016.
%
\bibitem{BM-TrC2-2016}
{\sc A. Bove and M. Mughetti},
{\it Analytic Hypoellipticity for Sums of Squares and the Treves Conjecture. 
\text{II}}, Anal. PDE \textbf{10} (2017), no. 7, 1613--1635.
%
%
\bibitem{Cordaro_Himonas-94}
{\sc P.D. Cordaro and A.A. Himonas},
{\it Global analytic hypo-ellipticity of a class of degenerate elliptic operator on the torus},
Math. Res. Lett., \textbf{1} no. 4 (1994), 501--510.
%
\bibitem{Cordaro_Himonas-98}
{\sc P.D. Cordaro and A.A. Himonas},
{\it Global analytic regularity for sum of squares of vector fields},
Trans. Amer. Math. Soc., \textbf{350} (1998), 4993--5001.
%
%
\bibitem{H67}
{\sc L. H\"ormander}, {\it Hypoelliptic second order differential equations}, 
Acta Math. {\bf 119} (1967), 147-171.
%
\bibitem{H80-book1}
{\sc L. H\"ormander},
{\it The analysis of linear partial differential operators. I. Distribution theory and Fourier analysis.} 
Grundlehren der Mathematischen Wissenschaften [Fundamental Principles of Mathematical Sciences], 
{\bf 256}. Springer-Verlag, Berlin, 1983. ix+391 pp. 
%
\bibitem{RS}
{\sc L. Preiss Rothschild and E. M. Stein}, 
{\it Hypoelliptic differential operators and nilpotent groups\/},
Acta Math.  {\bf 137} (1976), 247-320.
%
\bibitem{Tartakoff-96}
{\sc D.S. Tartakoff},
{\it Global (and local) analyticity for second order operators constructed
from rigid vector felds on products of tori},
Trans. Amer. Math. Soc., \textbf{348} (1996), 2577--2583.
%
\bibitem{Treves_Cj-2006}
{\sc F. Treves}, 
{\it On the analyticity of solutions of sums of squares of vector fields\/},  
Phase space analysis of partial differential equations, 315-329, Progr. Nonlinear Differential
Equations Appl., {\bf 69}, Birkh\"auser Boston, Boston, MA, 2006.
%

\end{thebibliography}

\end{document}